\documentclass[11pt, letterpaper]{article}

\usepackage[margin = 1in]{geometry}
\usepackage{amsmath,amsthm,amssymb,graphicx,mathtools,enumerate,xcolor}
\usepackage{caption}
\usepackage{subcaption}
\usepackage[shortlabels]{enumitem} 
\usepackage[english]{babel}
\usepackage[utf8]{inputenc}
\usepackage{enumitem}
\usepackage{hyperref}  
\usepackage{authblk}
\usepackage[
    doi=false,url=false,giveninits=true,hyperref]{biblatex}
\theoremstyle{plain}
\newtheorem{theorem}{Theorem}

\newtheorem{lemma}[theorem]{Lemma}
\newtheorem{conjecture}[theorem]{Conjecture}
\newtheorem{claim}[theorem]{Claim}
\newtheorem{observation}[theorem]{Observation}

\theoremstyle{definition}
\newtheorem{definition}[theorem]{Definition}
\newtheorem{setup}[theorem]{Setup}

\newtheorem*{procedure*undirected}{Random coloring procedure for undirected graphs}
\newtheorem*{procedure*}{Random coloring procedure for digraphs}
\numberwithin{equation}{section}

\newcommand{\sm}{\setminus}
\renewcommand{\subset}{\subseteq}

\newcommand{\D}{\Delta}
\newcommand{\K}{\mathrm{Keep}}
\newcommand{\R}{\mathrm{Retain}}

\newcommand{\Eq}{\mathrm{Eq}}

\renewcommand{\Pr}{\mathbf{Pr}}
\newcommand{\Exp}{\mathbf{E}}

\renewcommand{\le}{\leqslant}
\renewcommand{\ge}{\geqslant}
\renewcommand{\leq}{\leqslant}
\renewcommand{\geq}{\geqslant}

\newcommand{\Vq}{\mathrm{Vq}}

\renewcommand{\l}{\ell}
\newcommand{\p}{p}
\newcommand{\cL}{\mathcal{L}}
\newcommand{\Res}{\mathrm{Reserve}}

\title{The List Linear Arboricity of Digraphs}

\author[1]{Yueping Shi\thanks{Email: {\tt shiyp9@mail2.sysu.edu.cn}.} } 

\author[1]{Ping Hu\thanks{Email: {\tt huping9@mail.sysu.edu.cn}. Supported by National Key Research and Development Program of China (2021YFA1002100) and National Natural Science Foundation of China (12471337).} }

\affil[1]{School of Mathematics, Sun Yat-sen University, Guangzhou, 510275, China.}

\begin{document}

\maketitle
	
\begin{abstract}
	A \emph{(directed) linear forest} is a (di)graph whose components are (directed) paths. The \emph{linear arboricity} $la(F)$ of a (di)graph $F$ is the minimum number of (directed) linear forests required to decompose its edges. 
    Akiyama, Exoo, and Harary (1980) proposed the Linear Arboricity Conjecture that $la(G) \leq \left\lceil \frac{\D+1}{2}\right\rceil$ for any graph $G$ of maximum degree $\D$. The current best known bound, due to Lang and Postle (2023), establishes $la(G) \leq \frac{\D}{2} + 3\sqrt{\D} \log^4 \D$ for sufficiently large $\D$. And they proved this in the stronger list setting proposed by An and Wu.

    For a digraph $D$, let its maximum degree $\Delta(D)$ be the maximum of all in-degrees and out-degrees of its vertices.
    Nakayama and P\'{e}roche (1987) conjectured that $la(D) \leq \Delta(D)+1$ for every digraph $D$. 
    We extend Lang and Postle's result to digraphs with a matching error term.
    We show that  $la(D) \leq\D + 6\sqrt{\D} \log^4 \D$ for  any digraph $D$ with $\D = \Delta(D)$ sufficiently large.
    Moreover, we also establish this bound in the stronger list setting, where each arc $e \in A(D)$ is assigned a list of colors, and each arc is assigned a color from its list such that each color class forms a directed linear forest.
		
\end{abstract}

\section{Introduction}
In this paper, a \emph{digraph} is a finite loopless directed graph without parallel arcs (but possibly $2$-cycles) and an \emph{undirected graph} is also a finite and simple graph, unless otherwise stated. A \emph{multigraph} may contain parallel edges but no loops. The \emph{underlying multigraph} of a digraph $D$ is the graph $\overline{D}$ obtained by treating the arcs of $D$ as unordered edges. 
A \emph{linear forest} is a disjoint union of paths, and the \emph{linear arboricity} $la(G)$ of a graph $G$, defined by Harary~\cite{H70}, is the minimum number of linear forests needed to partition the edges of $G$.
In 1980, Akiyama, Exoo and Harary~\cite{AEH80} proposed the following linear arboricity conjecture.
\begin{conjecture}[Linear Arboricity Conjecture]\label{con:linear-arboricity}
	For every undirected graph $G$ with maximum degree $\D$, $\lceil \frac{\D}{2}\rceil \le la(G) \le\lceil \frac{\D+1}{2}\rceil$.
\end{conjecture}
The Linear Arboricity Conjecture has been extensively studied and verified for many classes of simple graphs (see e.g. \cite{CHY24,W24}). 
Since an edge decomposition into forests can be viewed as an edge coloring with each color class inducing a forest, it is natural to consider its list coloring version.

A list assignment to edges of $G$ is a function $L$ that to each edge $e \in E(G)$ assigns a set $L(e)$ of colors.
A \emph{linear $L$-coloring} of $G$ is a map $\phi$ defined on $E(G)$ such that $\phi(e)\in L(e)$ for every $e\in E(G)$, and for every color $c$ in the range of  $\phi$, $\phi^{-1}(c)$ induces a linear forest. The \emph{list linear arboricity} of $G$, denoted by $lla(G)$, is the minimum $k$ such that for every list assignment $L$ to edges of $G$ with $|L(e)|\ge k$ for every $e\in E(G)$, there exists a linear $L$-coloring of $G$. If all lists of $L$ are the same, then a linear $L$-coloring is the same as an edge decomposition into linear forests, so $la(G) \le lla(G)$.
Wu~\cite{W99} (Conjecture 6.2.5) conjectured that $lla(G) = la(G)$ for all graphs $G$. This was strengthened by An and Wu~\cite{AW99} to the List Linear Arboricity Conjecture. 

\begin{conjecture}[List Linear Arboricity Conjecture]\label{conj:LLAC}
	For every undirected graph $G$ with maximum degree $\D$, $\lceil\frac{\D}{2}\rceil \le la(G)=lla(G) \le \lceil \frac{\D+1}{2}\rceil $.
\end{conjecture}

The first asymptotic result of Conjecture~\ref{con:linear-arboricity} was obtained by Alon~\cite{Alo88} in 1988, stating that $la(G)\leq\frac{\D}{2} + O\left(\D \cdot \frac{\log \log \D}{\log \D}\right)$. The error term was improved to $O(\D^{2/3} \log^{1/3} \D)$ 
by Alon and Spencer~\cite{AS16} in 1992.  In 2020, Ferber, Fox and Jain~\cite{FFJ20} reduced this further to $O(\D^{2/3-\alpha})$, where $\alpha$ is a small positive constant. Recently, Christoph et al.~\cite{CDGHMM25} proved that $la(G)\leq\frac{\D}{2}+O(\log n)$ where $n$ is the number of vertices of $G$.
For Conjecture~\ref{conj:LLAC}, 
in 2021, Kim and Postle~\cite{KP17} proved that $lla(G)\leq \frac{\D}{2}(1+o(1))$. In 2023, Lang and Postle~\cite{LP23} proved the following theorem, which is currently the best known bound for both the Linear Arboricity Conjecture and the List Linear Arboricity Conjecture. 

\begin{theorem}[{Lang and Postle~\cite[Theorem 4]{LP23}}]\label{thm:Ordinary}
		Let $G$ be an undirected graph with maximum degree $\D$ sufficiently large, then $lla(G)\leq \frac{\D}{2} + 3\sqrt{\D} \log^4\D $.
\end{theorem}

Alon and Spencer~\cite{AS16} deduced $la(G)\leq\frac{\D}{2} + O(\D^{2/3} \log^{1/3} \D)$
by proving it for digraphs.
This motivates our investigation of the directed version of Conjecture~\ref{conj:LLAC}. Let $D=(V(D),A(D))$ be a digraph.
For every vertex $v \in V(D)$, let $N^+(v)$ ($N^-(v)$ respectively) denote the set of vertices $u$ such that $vu\in A(D)$ ($uv\in A(D)$ respectively). 
Then let $d^+(v)=|N^+(v)|$ and $d^-(v)=|N^-(v)|$. 
Let the maximum degree $\Delta(D)$ of $D$ be $\max \{\Delta^+(D),\Delta^-(D)\}$, where $\Delta^+(D) := \max_{v \in V} d^+(v)$ and $\Delta^-(D) := \max_{v \in V} d^-(v)$ denote the maximum out-degree and in-degree of $D$, respectively.
A \emph{directed linear forest} is a disjoint union of directed paths. The linear arboricity $la(D)$ of a digraph $D$, introduced by Nakayama and P\'{e}roche~\cite{NP87}, is the minimum number of directed linear forests required to partition the arc set of $D$. They conjectured that $la(D)\leqslant \Delta(D)+1$ for all digraphs $D$. 

For $d$-regular digraphs $D$ (where $d^+(v)=d^-(v)=d$ for all $v \in V(D)$), we have $la(D)\geqslant \left\lceil \frac{d|V(D)|}{|V(D)|-1}\right\rceil\geqslant d+1$, since each  forest spans at most $|V(D)|-1$ arcs. As any digraph can be embedded into a $d$-regular digraph by adding arcs and vertices, the Nakayama-P\'{e}roche conjecture is equivalent to the conjecture that $la(D)=d+1$ for $d$-regular digraphs $D$. 
In 2017, He et al.~\cite{HLBS17} discovered counterexamples: the symmetric complete digraphs $K_3^*$ and $K_5^*$ satisfy $la(K_3^*)=4=d+2$ ($d=2$) and $la(K_5^*)=6=d+2$ ($d=4$). Consequently, they conjectured that $la(D)=d+1$ for all $d$-regular digraphs $D$ except $K_3^*$ or $K_5^*$.

Similar to the graph case, a \emph{linear $L$-coloring} of a digraph $D$ is a map $\phi$ defined on $A(D)$ such that $\phi(e)\in L(e)$ for every $e\in A(D)$, and for every color $c$ in the range of $\phi$, $\phi^{-1}(c)$ induces a directed linear forest. The \emph{list linear arboricity} of $D$, denoted by $lla(D)$, is the minimum $k$ such that for every list assignment $L$ to arcs of $D$ with $|L(e)|\ge k$ for every $e\in A(D)$, there exists a linear $L$-coloring of $D$. 
Our main result generalizes Theorem~\ref{thm:Ordinary} to digraphs in the list setting with a matching error term.

\begin{theorem}\label{thm:main}
Let $D$ be a digraph with $\D=\Delta(D)$ sufficiently large, then $lla(D) \le \D+6 \sqrt{\D} \log^4\D $.
\end{theorem}
If applying our result to a $2d$-regular graph $G$, then we know any Eulerian orientation of $G$ admits an edge decomposition into at most $d+6 \sqrt{d} \log^4 d$ directed linear forests. 

The rest of this paper is organized as follows. In Section~\ref{sec:nibble}, we elaborate on our methodology and deduce Theorem~\ref{thm:main} from Claim~\ref{cla:nibbler}. In Section~\ref{sec:nibbler}, we present the proof of Claim~\ref{cla:nibbler}.

\section{Proof of the main result}\label{sec:nibble}
In this section, we give the proof of Theorem~\ref{thm:main}, which builds on Lang and Postle's framework~\cite{LP23}. And their approach builds on Molloy and Reed's pioneering work~\cite{MR13} for the List Coloring Conjecture. 
So we first outline Lang and Postle's proof structure before adapting it to digraphs. 

For undirected graph $G$ with sufficiently large maximum degree $\D$, Lang and Postle~\cite{LP23} proved that $lla(G)\leq \frac{\D}{2} + 3 \sqrt{\D} \log^4\D $ via the R\"odl's nibble method.
Let $\cL(e)$ be a list assignment with $\cL(e)\ge \frac{\D}{2} + 3 \sqrt{\D} \log^4\D $ for every edge $e$ of $G$.
They consider the list product  $\cL(e)\times \{1,2\}$ ($|\cL(e)\times \{1,2\}| \geq \D+ 6\sqrt{\D} \log^4 \D$) and say that  
colors $(c,1)$ and $(c,2)$ are \emph{twins}. 

Recall the following related definitions. 
A \emph{$t$-edge-coloring} of $G$ from a list assignment $L(e)$ is an assignment of colors to the edges of $G$ such that each edge receives a color from its list and no vertex is incident to more than $t$ edges of the same color.
An edge-coloring is \emph{partial} for $G$, if it is an edge coloring of a subgraph of $G$.
  
They first reserve some colors from lists for each edge, which will be used later. Then they use a random coloring procedure to iteratively color the edges to get a partial $1$-edge-coloring from $\cL(e)\times \{1,2\}$ without bicolored cycles between twin colors. 
Then by merging twin colors, this partial $1$-edge-coloring yields a partial $2$-edge-coloring of $G$ from $\cL(e)$ without monochromatic cycles. 
To prevent monochromatic cycles from emerging and analyze the random coloring procedure, the following definitions are needed and stated in the context of lists $L(e)\subset \cL(e)\times \{1,2\}$ of (yet unused) colors and a partial edge-coloring $\gamma$ of $G$.
\begin{definition}[Dangerous paths]
	Let $\gamma$ be a partial $1$-edge-coloring of $G$.
	Consider vertices $u,v$ and twin colors $c,c'$.
	A path $P$ is \emph{dangerous} for $(u,v,c)$ under $\gamma$, if 
	\begin{itemize}
		\item $P$ starts with $u$ and ends with $v$  and
		\item the edges of $P$ are colored alternately with $c$ and $c'$, starting and ending with $c'$.
	\end{itemize}
\end{definition}
So if $P$ is dangerous for $(u,v,c)$, then $uv$ cannot be colored with $c$.
To monitor dangerous paths for $(u,v,c)$, Lang and Postle monitor a family $P(u,v,c)$ of paths that are candidates for becoming dangerous.
	Given  a partial $1$-edge-coloring $\gamma$ of $G$ from some lists $L(e)$, the \emph{color neighbors} of a vertex $v$ and a color $c$, denoted by $N_{G,L,\gamma}(v,c)$, is the set of uncolored edges $f$ incident to $v$ with $c \in L(f)$.
For convenience, they denote the lists by $L(e)$ and color neighbors by $N(v,c)$.
The iterative random coloring procedure in~\cite{LP23} is carried out as follows.
\begin{procedure*undirected}\label{procedure}~
	\begin{enumerate}[(I)]
		\item \label{step:activate} \textbf{Edge activation.} {Activate} each uncolored edge with probability $\p=\frac{1}{4}$.
		
		\item \label{step:assign} \textbf{Assign colors.} Assign to each activated edge $e$ a color chosen uniformly at random from $L(e)$.
		
		\item \label{step:unassign-and-coinflip} \textbf{Resolve conflicts.}
		Uncolor every edge $e$, which is assigned the same color as an incident edge.
		If $e$ is assigned color $c$ and no neighbor was assigned $c$, then uncolor $e$ with probability $\Eq(e,c)$. 
		\item \label{step:update-lists-and-coinflip}  \textbf{Update lists.}
		For every vertex $v$ and color $c$, if $c$ is retained by some edge in $N(v,c)$, then remove $c$ from the lists $L(f)$ of all other edges $f \in N(v,c)$.
		If $c$ is not retained by an edge in $N(v,c)$, then with probability $\Vq(v,c)$ remove $c$ from the lists $L(f)$ of all edges $f \in N(v,c)$.
		\item \label{step:cycle-tweak}  \textbf{Prevent dangerous paths.}  For every uncolored edge $uv$ and color $c \in L(uv)$, if  there is a path in $P(u,v,c)$ that is dangerous for $(u,v,c)$ after step~\ref{step:assign}, then remove $c$ from $L(uv)$. 
		If $e$ was assigned $c$ also uncolor $uv$.
		(If $c$ was already removed in step~\ref{step:unassign-and-coinflip} or~\ref{step:update-lists-and-coinflip}, do nothing.)
	\end{enumerate}
\end{procedure*undirected}

At each iteration, $0 \leq \Eq(e,c), \Vq(v,c) \leq 1$, so these terms play the role of a coin flip probability.

They obtain a partial $1$-edge-coloring of $G$ from the list product $\cL(e)\times \{1,2\}$ by iteratively coloring the edges using this random coloring procedure until most edges of the graph have been colored. 
Then by mapping twin colors $(c,1)$ and $(c,2)$ to $c$, this partial $1$-edge-coloring translates to a partial $2$-edge-coloring of $G$ from $\cL(e)$ without any monochromatic cycles. 
Moreover, they can finish this by finding a $1$-edge-coloring for the remaining uncolored edges of $G$ from the colors that were reserved beforehand. 
So when combined with the partial $2$-edge-coloring, each color class is a linear forest.

\subsection{Modifying the random coloring procedure}\label{sec:Modifying the random coloring procedure}
We now return to our proof of Theorem~\ref{thm:main}.
Let $D$ be a digraph.
Assume $\D = \Delta(D)$ is sufficiently large.
Let $\cL$ be a list assignment with 
$|\cL(e)| \geq \D+ 6\sqrt{\D} \log^4 \D$ for every arc $e$ of $D$. 
Consider any sublist assignment $L(e) \subset \cL(e)$ to arcs of $D$.
An \emph{$(s,t)$-edge-coloring} of $D$ from $L(e)$ is an assignment of colors to the arcs of $D$ such that each arc $e$ receives a color from $L(e)$ and each vertex $v$ is incident to at most $s$ arcs oriented toward $v$ and at most $t$ arcs oriented away from $v$ of the same color.
Note that, if all lists are the same, then a decomposition of $D$ into directed linear forests is equivalent to a $(1,1)$-edge-coloring that does not induce any monochromatic directed cycles.

We adapt the proof of Lang and Postle to digraphs to prove Theorem~\ref{thm:main}. So we follow their framework to write our proof. 
The core difference is that, instead of using twin colors to produce a $2$-edge-coloring, we modify \ref{step:unassign-and-coinflip} \ref{step:update-lists-and-coinflip} of their random coloring procedure to guarantee that 
the outcome is a partial $(1,1)$-edge-coloring.

To prevent monochromatic directed cycles in the random coloring procedure, we will remove color $c$ from the list of an arc $uv$ whenever a directed path from $v$ to $u$ becomes monochromatic with $c$. To monitor such monochromatic paths, we monitor certain \emph{suspicious directed paths}, which are candidates for becoming monochromatic.

\begin{definition}[Suspicious directed paths]\label{def:suspicious1}
	Let $\gamma$ be a partial $(1,1)$-edge-coloring of $D$ from some lists $L(e)\subset \cL(e)$, and let $c$ be a color. 
	A directed path $P=(v_1,\dots,v_s)$ is \emph{suspicious} for $c$ if the following holds:
	\begin{itemize}
		\item If $v_iv_{i+1}$ is uncolored, then $c \in L(v_iv_{i+1})$.
		\item If $v_iv_{i+1}$ is colored, then $c = \gamma(v_iv_{i+1})$.
	\end{itemize}
\end{definition}
The \emph{uncolored length} of $P$ is the number of its uncolored arcs.
For a directed path from $v$ to $u$ that is suspicious for $c$, if its uncolored length is less than $\ell=2\log \D$, then we will monitor the whole path. Let $P(v,u,c;k)=P_{L,\gamma}(v,u,c;k)$ be the union of directed path from $v$ to $u$ that is suspicious for $c$ and of uncolored length $k$. If $k$ is at least $\ell$, then to better control the dependencies of events related to $uv$, we will forget $v$ and just monitor its later half with uncolored length exactly $\ell$ and first arc uncolored. Let $P(u,c;\ell)=P_{L,\gamma}(u,c;\ell)$ be the set of such paths, i.e., all directed path that ends at $u$, suspicious for $c$, first arc uncolored and of uncolored length $\ell$. 

To state and analyze our random coloring procedure, we also need the following definition.
\begin{definition}[color in-neighbors, out-neighbors]
	Given  a partial $(1,1)$-edge-coloring $\gamma$ of $D$ from some lists $L(e)\subset \cL(e)$, we define the \emph{color in-neighbors} (\emph{color out-neighbors} respectively) of a vertex $v$ and a color $c$, denoted by $N^-_{D,L,\gamma}(v,c)$ ($N^+_{D,L,\gamma}(v,c)$ respectively), to be the set of uncolored arcs $uv$ ($vu$ respectively) incident to $v$ with $c \in L(uv)$ ($L(vu)$ respectively). 
\end{definition}

Let $\gamma$ be a partial $(1,1)$-edge-coloring and $L(e)$ be the lists of $D$. 
For convenience, we denote color in-neighbors by $N^-(v,c)=N^-_{D,L,\gamma}(v,c)$ and color out-neighbors by $N^{+}(v,c)=N^{+}_{D,L,\gamma}(v,c)$.
The random coloring procedure involves two coin flips that will help us to bound some of the involved terms uniformly. See $Eq$ and $Vq$ defined in~\eqref{equ:Eq}, \eqref{equ:Vq^+} and \eqref{equ:Vq^-} later.

Now we color (some of) the arcs of $D$ iteratively using the following random process:
\begin{procedure*}\label{procedure1}~

The first two steps are the same as Lang and Postle's setup.	
	\begin{enumerate}[(I)]
		\item \label{step:activate1} \textbf{Arc activation.} Activate each uncolored arc with probability $\p=1/4$.
		
		\item \label{step:assign1} \textbf{Assign colors.} Assign to each activated arc $e$ a color chosen uniformly at random from $L(e)$.		
	\end{enumerate}
\end{procedure*}
 
The purpose of the following modification is to ensure that the partial edge coloring is a partial $(1,1)$-edge-coloring.
\begin{enumerate}[(I)] \addtocounter{enumi}{2}
		\item \label{step:unassign-and-coinflip1} \textbf{Resolve conflicts.}
	Uncolor every arc $e=uv$, which is assigned the same color as some arc in $N^{+}(u,c)\cup N^{-}(v,c) \backslash \{e\}$.
	If $e$ is assigned color $c$ and no arc in $N^{+}(u,c)\cup N^{-}(v,c) \backslash \{e\}$ is assigned $c$, then uncolor $e$ with probability $\Eq(e,c)$. 

If an arc is not uncolored, then we say that it \emph{retains} its color.
    
	\item \label{step:update-lists-and-coinflip1}  \textbf{Update lists.}
	For every vertex $v$ and color $c$, if $c$ is retained by some arc in $N^{+}(v,c)$ ($N^{-}(v,c)$ respectively), then remove $c$ from the lists $L(f)$ of all other arcs $f \in N^{+}(v,c)$ ($N^{-}(v,c)$ respectively).
	If $c$ is not retained by any arc in $N^{+}(v,c)$ ($N^{-}(v,c)$ respectively), then with probability $\Vq^{+}(v,c)$ ($\Vq^{-}(v,c)$ respectively)  remove $c$ from the lists $L(f)$ of all arcs $f \in N^{+}(v,c)$ ($N^{-}(v,c)$ respectively).
\end{enumerate}

The following modification deals with monochromatic directed paths. 
For every arc $uv$ uncolored in $\gamma$ and color $c \in L(uv)$, if a directed path from $v$ to $u$ becomes monochromatic with color $c$ after step~\ref{step:assign1}, then it has some arc uncolored before this step. 
So either there is some $1 \leq k <\ell= 2 \log \D$, such that $P \in P(v,u,c;k)$ or $P$ ends with a path in $P(u,c;\ell)$. 
Therefore it is enough to focus on monochromatic directed paths that correspond to 
\begin{align}\label{equ:monochromatic-paths}
	P(v,u,c) := \bigcup_{k=1}^{\ell-1} P(v,u,c;k) \cup P(u,c;\ell).
\end{align}
\begin{enumerate}[(I)] \addtocounter{enumi}{4}
	\item \label{step:cycle-tweak1}  \textbf{Prevent monochromatic direct cycles.}  For every arc $uv$ uncolored in $\gamma$ and color $c \in L(uv)$, if there is a directed path in $P(v,u,c)$ that becomes monochromatic with $c$ after step~\ref{step:assign1}, then remove $c$ from $L(uv)$. 
	If $uv$ was assigned $c$ also uncolor $uv$.
	(If $c$ was already removed in step~\ref{step:unassign-and-coinflip1} or~\ref{step:update-lists-and-coinflip1}, do nothing.)
\end{enumerate}

It is clear that if we start with $\gamma$ being an empty coloring, and run this random coloring procedure iteratively, we will get a partial $(1,1)$-edge-coloring without monochromatic directed cycles.
Before showing more details of this procedure, we first bound the number of suspicious directed paths. It is essentially the same proof as Lang and Postle's~\cite[Observation 8]{LP23}. For completeness, we include its proof here.

\begin{observation}\label{obs:number-suspicious}
Let $\gamma$ be a partial $(1,1)$-edge-coloring of $D$ from some lists $L(e)\subset \cL(e)$ at the beginning of the $i$-th iteration.
	Suppose there is a constant $N$ such that 
    $|N^-_{D,L,\gamma}(w,c)| \leq N$ for every vertex $w$ and color $c$.
	Then we have $|P_{L,\gamma}(v,u,c;k)|\leq N^{k-1}$ and $|P_{L,\gamma}(u,c;k)|\leq N^{k}$ for all vertices $u,v$, colors $c$ and integers $k$.
\end{observation}
\begin{proof}
We prove $|P_{L,\gamma}(u,c;k)|\leq N^{k}$ by induction on $k$.
Since $\gamma$ is a partial $(1,1)$-edge-coloring of $D$, there exists a unique edge maximal path $P$ ends at $u$ and colored with $c$. Let $P'\in P_{L,\gamma}(u,c;1)$ and let $xw$ be its first arc. By definition we know $xw\in N^-_{D,L,\gamma}(w,c)$ is uncolored and all other arcs of $P'$ are colored with $c$, so $w\in V(P)$. 
By step~\ref{step:update-lists-and-coinflip1} of our random coloring procedure, we know for each vertex $v$, if there exist an arc $uv$ colored with $c$, then $N^-_{D,L,\gamma}(v,c)=\emptyset$. Therefore $w$ must be the first vertex of $P$, otherwise $N^-_{D,L,\gamma}(w,c)=\emptyset$.
Then we have the induction base case $|P_{L,\gamma}(u,c;1)|\le |N^-_{D,L,\gamma}(w,c)| \le N$. 
For the induction step, note that for $k>1$, each path in $P_{L,\gamma}(u,c;k)$ is an extension of a path in $P_{L,\gamma}(u,c;k-1)$.
Using the same argument as before, we see that a path in $P_{L,\gamma}(u,c;k-1)$ extends to at most $N$ paths in $P_{L,\gamma}(u,c;k)$. So we have $|P_{L,\gamma}(u,c;k)|\leq N^{k}$.

For $P_{L,\gamma}(v,u,c;1)$, there exists a unique edge maximal path $Q$ starts at $v$ and colored with $c$. Similar to the above analysis on $P_{L,\gamma}(u,c;1)$, by step~\ref{step:update-lists-and-coinflip1} of our random coloring procedure, we know if $P'\in P_{L,\gamma}(v,u,c;1)$ with $xw$ uncolored, then $P'-xw = Q \cup P$. So $|P_{L,\gamma}(v,u,c;1)|\le 1$. For $k>1$, we see a path in $P_{L,\gamma}(v,u,c;k)$ is an extension of a path in $P_{L,\gamma}(u,c;k-1)$, and using the same argument as $P_{L,\gamma}(v,u,c;1)$, we see a path in $P_{L,\gamma}(u,c;k-1)$ can extend to at most $1$ path in $P_{L,\gamma}(v,u,c;k)$, since it must extend to the last vertex of $Q$. So $|P_{L,\gamma}(v,u,c;k)|\leq |P_{L,\gamma}(u,c;k-1)| \le N^{k-1}$.
\end{proof}

\subsection{Adapting former work}
In the following, we show more details of the iterative random coloring procedure. The Lov\'asz Local Lemma~\cite{EL75} plays an important roll on analyzing this procedure.

\begin{lemma}[{Lov\'asz Local Lemma}]\label{lem:LLL}
	Let $A_1,\dots,A_n$ be events in an arbitrary probability space.
	Suppose that each event $A_i$ is mutually independent of all but at most $d$ other events, and that $\Pr(A_i) \leq p$ for all $1 \leq i \leq n$.
	If $e p (d+1) \leq 1$, then $\Pr\left( \bigwedge_{i=1}^n \overline{A_i} \right) >0$.
\end{lemma}

It can be shown with the Lov\'asz Local Lemma that with positive probability, we can repeat the above random coloring procedure until most arcs of the digraph have been colored.
Once this is done, we can complete the coloring with the following lemma using a fresh set of colors that was reserved beforehand. Although the following lemma is stated for graphs in the original work, the proof provided there also holds for multigraphs. For completeness, we include its proof here.

\begin{lemma}[Finishing blow {\cite[Theorem 4.3]{MR13}}]\label{lem:finisher}
	Let $G$ be a multigraph with an assignment of lists $L(e)$ to the edges such that $|L(e)| \geq L$ for some constant $L$.
	Suppose that there is a constant $N$ such that for every edge $e$ and color $c$, there are at most $N$ edges $f$ incident with $e$ such that $c \in L(f)$.
	If $L \geq 8 N$, then $G$ has a $1$-edge-coloring from the lists $L$.
\end{lemma}

\begin{proof} 
For ease of exposition, we truncate each $L(e)$	so that it has exactly $L$ colors. Now, we assign to each edge $e$ a color chosen uniformly at random from $L(e)$. For each pair of incident edges $e_1$ and $e_2$, and color
$i \in L(e_1)\cap L(e_2)$, let $A_{i,e_1,e_2}$ be the event that both $e_1$ and $e_2$ are colored with $i$. Let $E$ be the set of all such events. We use the Lov\'asz Local Lemma to show that with positive probability none of the events in $E$ occur, i.e., $G$ has a $1$-edge-coloring.

Consider first the probability of $A_{i,e_1,e_2}$, clearly this is $1/L^2$. Consider next the dependency between events. $A_{i,e_1,e_2}$ depends only on the colors assigned to $e_1$ and $e_2$. letting $E_{e_1} = \{A_{j,e_1,f} | j \in L(e_1), f$ and $e_1$ are incident$\}$ and $E_{e_2} = \{A_{j,e_2,f} | j \in L(e_2), f$ and $e_2$ are incident$\}$, we see $A_{i,e_1,e_2}$ is mutually independent of $E-E_{e_1}-E_{e_2}$. Now, since $L(e_1)$ has exactly $L$ elements, and there are at most $N$ edges $f$ incident with $e_1$ of color $j$ for each $j \in L(e_1)$, we have $|E_{e_1}|\leq LN$. Similarly, $|E_{e_2}|\leq LN$. Thus each $A_{i,e_1,e_2}$ is mutually independent of a set of all but at most $2LN$ of the other events in $E$. Since $e\times (1/L^2 ) \times (2LN+1)\leq 1$, the Lov\'asz Local Lemma implies that a $1$-edge-coloring exists. This yields the desired result.	
\end{proof}

\paragraph{Reservations.}
We reserve a few colors, which will be used to finish the coloring using Lemma~\ref{lem:finisher}.
For each vertex $v$, we select a set of colors $\Res(v)$ from the lists of arcs incident to $v$.
The following lemma states that we can choose the sets $\Res(v)$ in a well-distributed way.
This lemma, though stated for graphs in the original source, has a proof in the original text that can be extended directly to digraphs. For completeness, we include its proof here.

\begin{lemma}[Reserve colors {\cite[Lemma 12]{MR00}},{\cite[Lemma 14.6]{MR13}}]\label{lem:reserved-colors}

Let $D$ be a digraph with sufficiently large maximum degree $\D$. Let $\mathcal{L}$ be an list assignments on arcs of $D$ such that 
$|\cL(e)| = \D+ 6\sqrt{\D} \log^4 \D$ for each arc $e$ of $D$.
Then for each vertex $v$, we can choose $\Res(v)$ such that for every arc $e=uv$ and colors $c_1 \in \Res(u), c_2 \in \Res(v)$, we have:
	\begin{enumerate}[\upshape (a)]
		\item $|\cL(e) \cap (\Res(u) \cup \Res(v))| \leq 3 \sqrt{\D} \log^4 \D$,
		\item $|\cL(e) \cap (\Res(u) \cap \Res(v))| \geq  \log^8 \D/2$, 
        \item $|\{w \in N^{+}(u) \colon c_1 \in \Res(w)\cap \mathcal{L}(uw)\}| \leq 2 \sqrt{\D} \log^4 \D$, \text{ and}
		\item $|\{w \in N^-(v) \colon c_2 \in \Res(w)\cap \mathcal{L}(wv) \}| \leq 2 \sqrt{\D} \log^4 \D$.
	\end{enumerate}
\end{lemma}

\begin{proof} 
For a vertex $v$, let $\cL(v)$ be the union of $\cL(f)$ over all arcs $f$ incident to $v$, so $|\cL(v)|\le 3\D^2$.
Then for each vertex $v$ and color $c\in \cL(v)$, we place $c$ into $\Res(v)$ with probability $p=log^4 \D/\sqrt{\D}$. For each arc $e=uv$ and colors $c_1\in \cL(u), c_2\in \cL(v)$, define $A_e,B_e,C_{u,c_1}$ and $D_{v,c_2}$ to be the event that $e$ violates condition (a),  $e$ violates condition (b), $u,c_1$ violate condition (c), and $v,c_2$ violate condition (d), respectively.

Note that 
$$\mathbb{E}(|\cL(e) \cap (\Res(u) \cup \Res(v))|) \leq |\cL(e)| \times 2p \approx 2\sqrt{\D} \log^4 \D,$$
$$\mathbb{E}(|\cL(e) \cap (\Res(u) \cap \Res(v))|) = |\cL(e)| \times p^2 \approx \log^8 \D,$$
$$\mathbb{E}(|\{w \in N^{+}(u) \colon c_1 \in \Res(w)\cap \mathcal{L}(uw)\}|) \leq \D \times p \approx \sqrt{\D} \log^4 \D,$$
$$\mathbb{E}(|\{w \in N^-(v) \colon c_2 \in \Res(w)\cap \mathcal{L}(wv) \}|) \leq \D \times p \approx \sqrt{\D} \log^4 \D.$$	

Thus, it is a straightforward application of the Chernoff Bound to show that the probability of any one event is (much) less than $e^{-log^2\D}$. Furthermore, a straightforward application of the Mutual Independence Principle verifies that each event $E$ is mutually independent of all those events which do not have a subscript at distance less than three from some subscript of $E$. 
Using the fact that for any vertex $v$, $C_{v,c}$ and $D_{v,c}$ is defined for fewer than $3 \D^2$ colors $c$, we see that each event is mutually independent of all but fewer than $\D^5$ other events. Therefore, our lemma follows from the Lov\'asz Local Lemma (Lemma~\ref{lem:LLL}).
\end{proof}

Given the sets $\{\Res(v)\}$, for each arc $e=uv$ we define
\begin{align}
	L_0(e) &= \mathcal{L}(e) \sm (\Res(u) \cup \Res(v)) \text{ and} \label{equ:def-Le} \\ 
	\Res(e)&= \mathcal{L}(e) \cap (\Res(u) \cap \Res(v)).\label{equ:def-Rservee}
\end{align} 

During the iterative coloring procedure, we color $D$ only from colors of $L_0(e)$.
In the final step, we color the remaining uncolored arcs with colors from $\Res(e)$ using Lemma~\ref{lem:finisher}. 
To apply Lemma~\ref{lem:finisher}, we need to monitor the relation between uncolored arcs and the reserved colors.
\begin{definition}\label{def:Rvc-tracking}
For a partial edge-coloring  $\gamma$ of $D$, let $R^{+}_{\gamma}(v,c)$ ($R^-_{\gamma}(v,c)$ respectively) be the set of uncolored arcs $vu$ ($uv$ respectively) with $c \in \Res(u)$.
\end{definition}

\paragraph{Parameters.}
Next, we define a number of parameters used to track the sizes of the lists of (still available) colors and other objects throughout the iterative coloring procedure.
Although the random coloring procedure is different, we use the same parameters as in \cite[Setup 13]{LP23}. Thus we can apply their lemmas directly in our analysis. 

\begin{setup}\label{Setup}
	Let $\p =1/4$, $\l= 2 \log \D$
    and define recursively:
	\begin{align*}
		L_0 &= \D+ {3 \sqrt{\D} \log^4 \D},	 & L_{i+1} &= L_i \cdot \K_i^2  - \sqrt{L_i} \log^2 \D,
		\\ N_0 &= \D, & N_{i+1} &= N_i \cdot  \K_i  \cdot \left(1- \p \R_i^2   \right) + \sqrt{N_i}  \log^2 \D,
		\\R_0 &= 2\sqrt{\D} \log^4 \D, & R_{i+1} &= R_i \cdot \left(1- \p \R_i^2   \right)   + \sqrt{R_i}\log^2 \D, 
		\\ 	\R_i &= \left(1-{\frac{\p}{L_i}}\right)^{N_i-1}, & \K_i &=  1- \p  \frac{N_i}{L_i} \cdot \R_i^2 \text{ and}
		\\ i_0 &= \min \{i \colon L_i < 3 \log^7 \D \}.
	\end{align*}
\end{setup}

At the beginning of the $(i+1)$-th iteration of the random coloring procedure, $L_i$ will be the list size of an arc, $N_i$ will be an upper bound on the size of color neighbors, and $R_i$ will be an upper bound on $|R^{+}_{\gamma}(v,c)|$ and $|R^-_{\gamma}(v,c)|$. From Lemma~\ref{lem:reserved-colors} we see this holds for $i=0$, and we will show the recurrence formula in Claim~\ref{cla:nibbler}.

For an arc $e=uv$, a vertex $v$ and a color $c$, the probability that no arc in $N^{+}(u,c)\cup N^{-}(v,c) \backslash \{e\}$ is assigned $c$, conditional on $e$ receiving $c$ is  $$P(e,c)= \left(1-  \frac{\p}{L_i} \right)^{| N^{+}(u,c)\backslash \{e\}|+| N^-(v,c)\backslash \{e\}|}.$$
Now we can give details of the probabilities used in our random coloring procedure.
Let
\begin{align}
	\Eq(e,c) &=  1- \frac{\R_i^2}{ P(e,c)},  \label{equ:Eq} \\
	\Vq^{+}(v,c) &= 1-\frac{\K_i}{1- \frac{\p}{L_i} |N^{+}(v,c)| \R_i^2} \text{ and } \label{equ:Vq^+} \\
	\Vq^-(v,c) &= 1-\frac{\K_i}{1- \frac{\p}{L_i} |N^-(v,c)| \R_i^2}. \label{equ:Vq^-}
\end{align}
We first give a sketch about what these parameters for. We will show that $Eq$ is used to make $\R_i^2$ be the probability that an arc retains an assigned color, and $Vq$ is used to make the probability that the list of an arc keeps a particular color be at least $\K_i^2$. 
Then we can show that after step~\ref{step:update-lists-and-coinflip1}, the expectation of the list size of an arc is at least $L_i \cdot \K_i^2$. For step~\ref{step:cycle-tweak1}, since the probability that an arc is activated and assigned a given color is $p/L_i$, the probability that a suspicious path with $k$ uncolored arcs becomes monochromatic is $(p/L_i)^k$. Then using the bounds on the number of suspicious paths in Observation~\ref{obs:number-suspicious}, we will show that step~\ref{step:cycle-tweak1} remove $O(1)$ colors from the list of an arc in expectation. Then we will show that the list size of an arc is highly concentrated around its expectation, so it is at least $L_i \cdot \K_i^2  - \sqrt{L_i} \log^2 \D$ after the $(i+1)$-th iteration, where $\sqrt{L_i} \log^2 \D$ is the concentration error term. So we can take $L_{i+1} = L_i \cdot \K_i^2  - \sqrt{L_i} \log^2 \D$. The analyses for $N_i$ and $R_i$ are similar.

The following lemma provides bounds for the parameters $L_i, N_i, R_i$.

\begin{lemma}[{Lang and Postle~\cite[Lemma 14]{LP23}}, after {Molloy and Reed~\cite[Lemma 7]{MR00}}]\label{lem:sizes}
	Suppose that $\D$ is sufficiency large and consider $R_{i}, L_{i}, N_{i},i_0$ as in Setup~\ref{Setup}.
	Then it follows that $L_{i_0}, N_{i_0}, R_{i_0}  > \log^7 \D$, $R_{i_0} \leq 3 \log^{7.5} \D$, $R_{i}/L_{i} \leq \log \D$, and $L_i > N_i > L_i/2$ for each $i \leq i_0$.
\end{lemma}

\paragraph{Iterations.}
As noted before, we will use $L_i,N_i,R_i$ to track objects throughout the iterative coloring procedure.
This will be stated in the next claim.
The following definitions help us with this statement.
A partial $(1,1)$-edge-coloring $\gamma$ of $D$ from $\cL(e)$ is called \emph{acyclic} if it does not induce any monochromatic directed cycles.
We say that lists $L(e) \subset \cL(e)$ are \emph{$\gamma$-compatible} if both of the following hold:

\begin{itemize}
	\item if an arc $uv$ is colored by $\gamma(uv)=c$, then $c \notin L(e)$ for all other arcs $e\in N^{+}(u)\cup N^-(v)$. 
	\item if an arc $e=uv$ is not colored by $\gamma$ and there exists a monochromatic directed path from $v$ to $u$ colored with $c$ under $\gamma$, then $c\not\in L(e)$.
\end{itemize}
We say another edge-coloring $\gamma'$ of $D$ \emph{extends} $\gamma$, if every arc colored by $\gamma$ is colored by $\gamma'$ in the same way.
Similarly, the lists $L(e)$ \emph{extend} lists $L'(e) \subset \cL(e)$, if $L(e) \subset L'(e)$ for all arcs $e$ of $D$.

\begin{claim}[Single coloring step]\label{cla:nibbler}
	Given Setup~\ref{Setup} and $\D$ large enough, the following holds for every $0 \leq i < i_0$.
	Suppose there is an acyclic $(1,1)$-edge-coloring $\gamma_i$ of $D$ and $\gamma_i$-compatible lists $L_i(e) \subset \cL(e)$ with the following properties:
	For every uncolored arc $e=uv$ and color $c \in L_i(e)$, we have
	\begin{enumerate}[\upshape (a)]
		\item \label{itm:list-size-before}  $|L_i(e)| = L_i$, 
		\item \label{itm:color-neighbors-size-before}  $|N^{+}_{D,L_i,\gamma_i}(u,c)| \leq N_i$, $|N^-_{D,L_i,\gamma_i}(v,c)| \leq N_i$, 
		\item \label{itm:reserved-size-before}  $|R^{+}_{\gamma_i}(u,c)| \leq R_i$, $|R^-_{\gamma_i}(v,c)| \leq R_i$, and
		\item \label{itm:cycle-conflict-before} there is no monochromatic directed path from $v$ to $u$ colored with $c$ under $\gamma_i$.
	\end{enumerate}
	
	Then it holds that: there exist an acyclic $(1,1)$-edge-coloring $\gamma_{i+1}$ of $D$ from $\cL(e)$ and $\gamma_{i+1}$-compatible lists $L_{i+1}(e)$ such that $L_{i+1}(e),\gamma_{i+1}$  {extend} $L_{i}(e),\gamma_{i}$.
	Moreover, for every uncolored arc $e=uv$ and color $c \in L_{i+1}(e)$, we have 
	\begin{enumerate}[\upshape (a$'$)]
		\item \label{itm:list-size-after} $|L_{i+1}(e)| = L_{i+1}$, 
		\item $|N^{+}_{D,L_{i+1},\gamma_{i+1}}(v,c)| \leq N_{i+1}$, $|N^-_{D,L_{i+1},\gamma_{i+1}}(v,c)| \leq N_{i+1}$,  
		\item \label{itm:reserved-size-after} $|R^{+}_{\gamma_{i+1}}(v,c)| \leq R_{i+1}$, $|R^-_{\gamma_{i+1}}(v,c)| \leq R_{i+1}$, and
		\item \label{itm:cycle-conflict-after}there is no monochromatic directed path from $v$ to $u$ colored with $c$ under $\gamma_{i+1}$.
	\end{enumerate}
\end{claim}

Now we first finish the proof of Theorem~\ref{thm:main}. And the proof of Claim~\ref{cla:nibbler} will be in Section~\ref{sec:nibbler}.

\paragraph{Finishing the proof.}
We iteratively apply Claim~\ref{cla:nibbler} starting with $L_0(e)$, defined in~\eqref{equ:def-Le}, and the empty coloring of $D$, denoted by $\gamma_0$.
This yields a sequences of extensions $(L_i,\gamma_i)$ satisfying the outcome of Claim~\ref{cla:nibbler}.
In particular, each $\gamma_{i}$ is a $(1,1)$-edge-coloring, which does not contain any monochromatic directed cycles.
After $i_0$ steps, Lemma~\ref{lem:sizes} yields that for every arc $e$ and color $c \in \Res(e)$, defined in~\eqref{equ:def-Rservee}, there are at most $4R_{i_0}\leq 12\log^{7.5} \D$ arcs $f$ incident to $e$ for which $c \in \Res(f)$.
On the other hand $|\Res(e)| \geq  (\log^8 \D)/2$ by Lemma~\ref{lem:reserved-colors}.
So using Lemma~\ref{lem:finisher}, we can color all the remaining uncolored edges of the underlying graph of $D$ with a $1$-edge-coloring from the lists $\Res(e)$. 
Note that if an arc $f$ is incident to $e$, then $L_0(e)\cap \Res(f)=\emptyset$, 
so no new monochromatic directed cycles are generated.
This completes the proof of Theorem~\ref{thm:main}. \hfill\qed

\section{Proof of Claim~\ref{cla:nibbler}}\label{sec:nibbler}
In this section we use the random coloring procedure for digraphs in Section~\ref{sec:Modifying the random coloring procedure} to prove Claim~\ref{cla:nibbler}. The proof follows the argument for Claim~15 in \cite[Section 3]{LP23}. We will show that if we run the random coloring procedure with the partial $(1,1)$-edge-coloring $\gamma_i$ and the $\gamma_i$-compatible lists $L_i(e)$, then with positive probability, we will get our desired $\gamma_{i+1}$ and $L_{i+1}(e)$.

We first analyze step~\ref{step:activate1} to~\ref{step:update-lists-and-coinflip1} of the random coloring procedure.
For convenience, we denote $\gamma=\gamma_i$ and $L(e)=L_i(e)$.
We also denote the partial $(1,1)$-edge-coloring and lists of colors after step~\ref{step:update-lists-and-coinflip1} of the procedure by  $\gamma'$ and $L'(e)$.
And we also abbreviate $N^{+}(v,c)=N^{+}_{D,L,\gamma}(v,c)$, $N^-(v,c)=N^-_{D,L,\gamma}(v,c)$, $R^{+}(v,c)=R^{+}_{\gamma}(v,c)$, $R^-(v,c)=R^-_{\gamma}(v,c)$, $N'^{+}(v,c)=N^{+}_{D,L',\gamma'}(v,c)$, $N'^-(v,c)=N^-_{D,L',\gamma'}(v,c)$, $R'^{+}(v,c)=R^{+}_{\gamma'}(v,c)$ and $R'^-(v,c)=R^-_{\gamma'}(v,c)$.
A concentration analysis shows that the sizes of the random variables $L'(e)$, $N'^{+}(v,c)$, $N'^-(v,c)$, $R'^{+}(v,c)$ and $R'^-(v,c)$ are (individually) bounded with high probability.
Since our parameters in Setup~\ref{Setup} are the same as in \cite[Setup 13]{LP23}, the proof is similar. The details of the proof of Claim~\ref{cla:MR-edge-bounds} are spelled out in Appendix~\ref{sec:MR-edge-bounds}.

\begin{claim}[{Lang and Postle~\cite[Claim 16]{LP23}}]\label{cla:MR-edge-bounds}
	Fix an arc $e$, a vertex $v$ and a color $c$.
	It holds with probability at least $1-\D^{-10\log \D}$ that
	\begin{align*}
		|L'(e)| &\geq L_i \cdot \K_i^2 - \tfrac{1}{2} \sqrt{L_i} \log^2 \D,
		\\ |N'^{+}(v,c)| &\leq N_i \cdot  \K_i  \cdot \left(1- \p \R_i^2   \right) + \tfrac{1}{2}\sqrt{N_i} \log^2 \D ,
		\\ |N'^-(v,c)| &\leq N_i \cdot  \K_i  \cdot \left(1- \p \R_i^2   \right) + \tfrac{1}{2}\sqrt{N_i} \log^2 \D ,
		\\ |R'^{+}(v,c)| &\leq R_i \cdot  \left(1- \p \R_i^2   \right) + \tfrac{1}{2} \sqrt{R_i} \log^2 \D,  \text{ and}
		\\ |R'^-(v,c)| &\leq R_i \cdot  \left(1- \p \R_i^2   \right) + \tfrac{1}{2} \sqrt{R_i} \log^2 \D.
	\end{align*}
\end{claim}

Now we analyze step~\ref{step:cycle-tweak1}.
Consider an arc $e=uv$ and a color $c$ such that there is a monochromatic directed path $P$ from $v$ to $u$ colored with $c$ under $\gamma'$.
By condition~\ref{itm:cycle-conflict-before} of Claim~\ref{cla:nibbler}, $P$ must have an uncolored arc under $\gamma$.
Hence we can focus on monochromatic directed paths that correspond to $P(v,u,c)$, as defined in (\ref{equ:monochromatic-paths}).
We denote the coloring after step~\ref{step:cycle-tweak1} by $\gamma''$ and the lists of colors by $L''$.
Clearly $(\gamma'',L'')$ is an extension of $(\gamma,L)$.
Moreover, $\gamma''$ is still a $(1,1)$-edge-coloring and also satisfies~\ref{itm:cycle-conflict-after} of Claim~\ref{cla:nibbler}.
It remains to show that $\gamma''$ satisfies~\ref{itm:list-size-after}--\ref{itm:reserved-size-after}, so we need to analyze the impact of step~\ref{step:cycle-tweak1} on the random variables $L',N',R'$.
The following random variables help us to track these deviations.

\begin{definition}
	Consider an (in $\gamma$) uncolored arc $uv$ and a color $c$.
	\begin{itemize}
		\item Let $X(uv)$ be the set of colors $c\in L(uv)$ for which there is a directed path in $P(v,u,c)$ that becomes monochromatic with $c$  after step~\ref{step:assign1}.
		\item Let $Y^{+}(v,c)$ ($Y^- (v,c)$ respectively) be the set of arcs $vw \in N^{+}(v,c)$ ($wv \in N^-(v,c)$ respectively) for which there is a directed path in $P(w,v,c)$ ($P(v,w,c)$ respectively) that becomes monochromatic with $c$  after step~\ref{step:assign1}.
		\item Let $Z^{+}(v,c)$ ($Z^- (v,c)$ respectively) be the set of arcs $vw$ ($wv$ respectively) with $c \in \Res(w)$, for which there is a directed path in $P(w,v,c)$ ($P(v,w,c)$ respectively) that becomes monochromatic  with $c$ after step~\ref{step:assign1}.
	\end{itemize}
\end{definition}

Denote by  $N''^{+}(v,c)$, $N''^-(v,c)$, $R''^{+}(v,c)$, $R''^-(v,c)$ the objects obtained from $N'^{+}(v,c)$, $N'^-(v,c)$, $R'^{+}(v,c)$, $R'^-(v,c)$ after step~\ref{step:cycle-tweak1}.

Note that
\begin{align*}
	|L''(e)| &\geq |L'(e)| - |X(e)|,
	\\	|N''^{+}(v,c)| &\leq |N'^{+}(v,c)| + |Y^{+}(v,c)|, 
	\\	|N''^-(v,c)| &\leq |N'^-(v,c)| + |Y^-(v,c)|, 
	\\|	R''^{+}(v,c) |&\leq |R'^{+}(v,c)| + |Z^{+}(v,c)|,   \text{ and }
	\\|	R''^-(v,c) |&\leq |R'^-(v,c)| + |Z^-(v,c)|.
\end{align*}
The next claim shows that the sizes of $X(e)$, $Y^{+}(v,c)$, $Y^-(v,c)$, $Z^{+}(v,c)$ and $Z^-(v,c)$ are bounded with high probability.
\begin{claim}[{Lang and Postle~\cite[Claim 18]{LP23}}]\label{cla:cycle-tweak}
	Fix an arc $e$, a vertex $v$ and a color $c$.
	It holds with probability at least $1-\D^{-10\log \D}$ that 
	\begin{align*}
		|X(e)| &\leq  \tfrac{1}{2} \sqrt{L_i} \log^2 \D,
		\\ |Y^{+}(v,c)| &\leq \tfrac{1}{2} \sqrt{N_i} \log^2 \D,
		\\ |Y^-(v,c)| &\leq \tfrac{1}{2} \sqrt{N_i} \log^2 \D,
		\\ |Z^{+}(v,c)| &\leq \tfrac{1}{2} \sqrt{R_i} \log^2 \D,  \text{ and}
		\\ |Z^-(v,c)| &\leq \tfrac{1}{2} \sqrt{R_i} \log^2 \D.
	\end{align*}
\end{claim}
Using the Lov\'asz Local Lemma, our final claim shows that the above bounds on random variables can hold simultaneously. The proofs of Claims~\ref{cla:cycle-tweak} and \ref{cla:simultaneously} are almost the same as the cycle modification analysis in \cite{LP23}. There are only a few differences due to the extension to digraphs. The details of the proof are spelled out in Appendix~\ref{sec:cycle}.

\begin{claim}[{Lang and Postle~\cite[Claim 19]{LP23}}]\label{cla:simultaneously}
	With positive probability, the conclusions of Claim~\ref{cla:MR-edge-bounds} and~\ref{cla:cycle-tweak} hold for all arcs $e$, vertices $v$ and colors $c$.
\end{claim}

Now we finish the proof of Claim~\ref{cla:nibbler}.

By Claim~\ref{cla:simultaneously}, we can choose the color assignments such that
\begin{align*}
	|L''(e)| &\geq \left( L_i \cdot \K_i^2 - \tfrac{1}{2} \sqrt{L_i} \log^2 \D\right) - \left(\tfrac{1}{2} \sqrt{L_i} \log^2 \D\right) = L_{i+1},
	\\	|N''^{+}(v,c)| &\leq   \left(N_i \cdot  \K_i  \cdot \left(1- \p \R_i^2   \right) + \tfrac{1}{2}\sqrt{N_i} \log^2 \D \right)+ \left(\tfrac{1}{2} \sqrt{N_i} \log^2 \D\right) = N_{i+1}, 
	\\	|N''^-(v,c)| &\leq   \left(N_i \cdot  \K_i  \cdot \left(1- \p \R_i^2   \right) + \tfrac{1}{2}\sqrt{N_i} \log^2 \D \right)+ \left(\tfrac{1}{2} \sqrt{N_i} \log^2 \D\right) = N_{i+1}, 
	\\|R''^{+}(v,c)| & \leq \left(R_i \cdot  \left(1- \p \R_i^2   \right) + \tfrac{1}{2} \sqrt{R_i} \log^2 \D \right) + \left(\tfrac{1}{2} \sqrt{R_i} \log^2 \D\right) = R_{i+1}, \text{ and }
	\\|R''^-(v,c)| & \leq \left(R_i \cdot  \left(1- \p \R_i^2   \right) + \tfrac{1}{2} \sqrt{R_i} \log^2 \D \right) + \left(\tfrac{1}{2} \sqrt{R_i} \log^2 \D\right) = R_{i+1}.
\end{align*}
Take $L_{i+1}(e)=L''(e)$ and $\gamma_{i+1} = \gamma''$.
Note that $\gamma_{i+1}$ is acyclic and the lists $L_{i+1}(e)$ are $\gamma_{i+1}$-compatible due to steps~\ref{step:update-lists-and-coinflip1} and~\ref{step:cycle-tweak1} of the procedure. Note that  $|L_{i+1}(e)|\ge L_{i+1}$, so to ensure that $|L_{i+1}(e)|=L_{i+1}$, we simply delete $|L_{i+1}(e)|-L_{i+1}$ colors from every list $L_{i+1}(e)$.
This completes the proof of Claim~\ref{cla:nibbler}. \hfill\qed

\printbibliography

\newpage
\appendix

\section{Proof of Claim~\ref{cla:MR-edge-bounds}}\label{sec:MR-edge-bounds}
The proof of Claim~\ref{cla:MR-edge-bounds} can be found in Appendix B of Lang and Postle~\cite{LP23} or the monograph of Molloy and Reed~\cite{MR13} for graphs. Although our settings are for digraphs, there are only a few changes in the proof.
The following arguments are therefore almost  identical to the ones of Lang and Postle.

We prove  Claim~\ref{cla:MR-edge-bounds} using a concentration analysis.
We first bound the expectations of the random variables in Claim~\ref{cla:MR-edge-bounds} (Subsection~\ref{sec:MR-expectation}), then show that these random variables are highly concentrated around their expectations (Subsection~\ref{sec:MR-concentration}).

\subsection{Expectation} 
\label{sec:MR-expectation}
The following three claims bound the expectations of $|R'^{+}(v,c)|$, $|R'^-(v,c)|$, $|L'(e)|$, $|N'^{+}(v,c)|$ and $|N'^-(v,c)|$.
\begin{claim}[{\cite[Claim 23]{LP23}}]
	We have	$\Exp(|R'^{+}(v,c)|) $, $\Exp(|R'^-(v,c)|) \le  R_i \cdot  (1-  \p\R_i^2)$ for every vertex $v$ and color $c$.
\end{claim}
\begin{proof}
	For any arc $e=wu$,
	suppose in the random coloring procedure, $e$ is activated in step~\ref{step:activate1} and assigned a color $x \in L(e)$ in step~\ref{step:assign1}.
	Recall the definition of $\R_i$, $\K_i$ in  Setup~\ref{Setup}  and $\Eq(e,c)$ in~\eqref{equ:Eq}.
	It follows that the probability that $e$ retains $x$ after step~\ref{step:unassign-and-coinflip1} is
	\begin{align}\label{eq:retain}
		\left(1-  \frac{\p}{L_i} \right)^{| N^-(u,x)|+| N^{+}(w,x)|-2} \cdot (1-\Eq(e,x)) =  \R_i^2,
	\end{align}
	where each factor $1-  \frac{\p}{L_i}$ represents the probability that none of $N^{-}(u,x)\cup N^{+}(w,x) \backslash \{e\}$ is assigned $x$.
	Hence, if $e$ is uncolored at the beginning, then the probability that $e$ is still uncolored after step~\ref{step:unassign-and-coinflip1} is $1-\p\R_i^2$.
    Note that $R'^{+}(v,c)$ and $R'^-(v,c)$ are the sets of arcs in $R^{+}(v,c)$ and $R^-(v,c)$, respectively, that are still uncolored after step~\ref{step:unassign-and-coinflip1}.
	Therefore the claim follows from linearity of expectation.
\end{proof}
\begin{claim}[{\cite[Claim 24]{LP23}}]
	We have	$\Exp(|L'(e)|) \geq L_i \cdot \K_i^2 $ for every arc $e$ and color $c$.
\end{claim}
\begin{proof}
	Consider an arc $e = uv$ and a color $c \in L_i(e)$. 
	We show that the probability that $c \in L'(e)$ is at least $\K_i^2$.
	From this, the claim follows from $|L_i(e)|=L_i$ and linearity of expectation.
	
	Let $K_u$, $K_v$ be the events that no arc in $N^{+}(u,c) \sm \{e\}$, respectively $N^-(v,c) \sm \{e\}$
	retains $c$ after step~\ref{step:unassign-and-coinflip1} of the procedure. 
	It turns out that the simplest way to compute $\Pr(K_u \cap K_v)$ is
	through the indirect route of computing $\Pr(\overline{K_u} \cup \overline{K_v})$ which is equal to
	$1 - \Pr(K_u \cap K_v)$. 
	Now, by the most basic case of the Inclusion-Exclusion Principle, 
	$\Pr(\overline{K_u} \cup \overline{K_v})$ is equal to $\Pr(\overline{K_u})+ \Pr(\overline{K_v})- \Pr(\overline{K_u} \cap \overline{K_v})$.
	Following the computation in \eqref{eq:retain} we know that $$\Pr(\overline{K_u}) + \Pr(\overline{K_v}) =  (|N^{+}(u,c)|+|N^-(v,c)|-2)  \frac{\p}{L_i} \R_i^2,$$ so we just need to
	bound $\Pr(\overline{K_u} \cap \overline{K_v})$.
	
	$\Pr(\overline{K_u} \cap \overline{K_v})$ is the probability that there is some pair of arcs $e_1 = uw \in N^{+}(u,c)\sm \{e\}$, and $e_2 = xv \in  N^-(v,c)\sm \{e\}$ such that $e_1$ and $e_2$ are
	both activated, receive and retain $c$ during the procedure ($w=x$ is possible).
	Now, for each such pair, let $R_{e_1 ,e_2}$ be the event that $e_1$ and $e_2$ both retain $c$ after step~\ref{step:unassign-and-coinflip1}.
	It follows that
	\begin{align*}
		\Pr(R_{e_1 ,e_2})= \left(\frac{\p}{L_i}\right)^2 \left(1- \frac{\p}{L_i}\right)^{|N^{+}(u,c) \cup N^-(w,c) \cup N^-(v,c)  \cup N^{+}(x,c)|-2} (1-\Eq(e_1,c)) (1-\Eq(e_2,c)).
	\end{align*}
	Since \[\left(1- \frac{\p}{L_i}\right)^{|N^{+}(u,c) \cup N^-(w,c) \cup N^-(v,c)  \cup N^{+}(x,c)|-2} > \left(1- \frac{\p}{L_i}\right)^{|N^{+}(u,c)| + |N^-(w,c)| -2 + |N^-(v,c)|  + |N^{+}(x,c)|-2},
    \]
	we obtain:
	\begin{align*}
		\Pr(R_{e_1 ,e_2}) >   \left(\frac{\p}{L_i}  \R_i^2 \right)^2,
	\end{align*}
	where the computation in \eqref{eq:retain} is used again.
    
	For any two distinct pairs $( e_1, e_2)$ and $( e_1',e_2')$, it is impossible for $R_{e_1 ,e_2}$
	and $R_{e_1' ,e_2'}$ to both hold. 
	Therefore the probability that $R_{e_1 ,e_2}$ holds for at least one pair, is equal to the sum over all pairs $e_1, e_2$ of $\Pr(R_{e_1 ,e_2})$, which by the above computation yields:
	\begin{align*}
		\Pr(\overline{K_u} \cap \overline{K_v}) &\geq  \left((|N^{+}(u,c)|-1)(|N^-(v,c)|-1)\right)  \left(\frac{\p}{L_i}  \R_i^2 \right)^2.
	\end{align*}
	Combining this with our bound  on $\Pr(\overline{K_u})+ \Pr(\overline{K_v})$, we see that:
	\begin{align*}
		\Pr( {K_u} \cap {K_v}) &\geq 1 -   (|N^{+}(u,c)|+|N^-(v,c)|-2)  \frac{\p}{L_i} \R_i^2  
		\\&\quad+ (|N^{+}(u,c)|-1)(|N^-(v,c)|-1)  \left(\frac{\p}{L_i}  \R_i^2 \right)^2
		\\&\geq \left(1-  \frac{ \p}{L_i} |N^{+}(u,c)| \R_i^2 \right) \left(1-  \frac{\p}{L_i}|N^-(v,c)|  \R_i^2 \right) . 
	\end{align*}
	Recall the definition of $\Vq^{+}$ in~\eqref{equ:Vq^+} and $\Vq^-$ in~\eqref{equ:Vq^-}.
	It follows  that
	\begin{align*}
		\Pr( c \in L'(e)) &=\Pr(K_u \cap K_v)(1-\Vq^{+}(u,c))  (1-\Vq^-(v,c))  
		\\&\geq \left(1 - \p \frac{|N^{+}(u,c)|}{L_i} \R_i^2\right)(1-\Vq^{+}(u,c)) \cdot  \left(1 - \p \frac{|N^-(v,c)|}{L_i} \R_i^2\right)(1-\Vq^-(v,c)) 
		\\&= \K_i^2,  
	\end{align*}
	as desired.
\end{proof}
It remains to deal with $|N'^{+}(v,c)|$ and $|N'^-(v,c)|$.
However, these random variables are not concentrated around their expectations, 
since if color $c$ is assigned to one of the arcs of $N^{+}(v,c)$ ($N^-(v,c)$ respectively), then $|N'^{+}(v,c)|$ ($|N'^-(v,c)|$ respectively) drops immediately to zero.
We will therefore carry out the expectation and concentration analysis for the size of a different set $N^{*+}(v,c)$ ($N^{*-}(v,c)$ respectively), which ignores the assignments of $c$ to arcs of $N^{+}(v,c)$ ($N^-(v,c)$ respectively). 
More precisely, let $N^{*+}(v,c)$ ($N^{*-}(v,c)$ respectively) be the set of arcs $vu \in N^{+}(v,c)$  ($uv \in N^{-}(v,c)$ respectively) such that $vu$ ($uv$ respectively) is still uncolored after step~\ref{step:unassign-and-coinflip1}, 
and in step~\ref{step:update-lists-and-coinflip1},
$c$ is not removed from $L(e)$ for any arc $e\in N^-(u,c)$ ($N^+(u,c)$ respectively).
Since $N'^{+}(v,c) \subset N^{*+}(v,c)$ and $N'^-(v,c) \subset N^{*-}(v,c)$, it suffices to focus the analysis on $|N^{*+}(v,c)|$ and $|N^{*-}(v,c)|$ to give upper bounds on $|N'^{+}(v,c)|$ and $|N'^-(v,c)|$.
\begin{claim}[{\cite[Claim 25]{LP23}}]
	We have	$\Exp(|N^{*+}(v,c)|)$, $\Exp(|N^{*-}(v,c)|) \leq N_i \cdot  \K_i  \cdot \left(1- \p \R_i^2   \right) +1$ for every vertex $v$ and color $c$.
\end{claim}
\begin{proof}
	We prove this for $N^{*+}(v,c)$. The proof for $N^{*-}(v,c)$ is strictly symmetric and requires no additional arguments.
	We will show that for each $e = vu$ in $N^{+}(v,c)$, we have: $\Pr(e \in N^{*+}(v,c)) \leq \K_i  \cdot \left(1- \p \R_i^2   \right) + \frac{1}{L_i}$.
	As before this implies the result by linearity of expectation.
	
	We define $A$ to be the event that $e$ is still uncolored after step~\ref{step:unassign-and-coinflip1} and $B$ to be the event that no arc in $N^-(u,c)$ retains $c$.
	We wish to bound $\Pr(A \cap B)$.
	Once again, we proceed in an indirect way, and focus instead on $\Pr(\overline{A} \cap \overline{B})$, showing that $\Pr(\overline{A} \cap \overline{B}) \leq \p^2 \frac{|N^-(u,c)|}{L_i} \R_i^4+ \frac{1}{L_i}$, thus implying
	\begin{align*}
		\Pr(A \cap B) &= \Pr(A) - \Pr(\overline{B}) + \Pr(\overline{A} \cap \overline{B}) 
		\\ &\leq \left(1-\p \R_i^2 \right) - \p \frac{|N^-(u,c)|}{L_i}\R_i^2 + \p^2 \frac{|N^-(u,c)|}{L_i} \R_i^4+ \frac{1}{L_i}
		\\ &= \left(1-\p \frac{|N^-(u,c)|}{L_i}\R_i^2 \right) \left(1-\p \R_i^2 \right) + \frac{1}{L_i}.
	\end{align*}
	Therefore 
	\begin{align*}
		\Pr(e \in N^{*+}(v,c)) = \Pr(A \cap B) (1-\Vq^-(u,c)) \leq \K_i \left(1-\p \R_i^2 \right) + \frac{1}{L_i}.
	\end{align*}
	
	For each color $x \in L(e)$ and arc $f=wu$ in $N^-(u,c) \sm \{e\}$, we define $Z(x,f)$ to be the event that $e$ retains $x$ and $f$ retains $c$.
	For each $x\neq c$, we have
	\begin{align*}
		\Pr(Z(x,f)) &= \left(\frac{\p}{L_i}\right)^2 \left(1-\frac{2\p}{L_i}\right)^{\left|\left[ (N^-(u,x)\cap N^-(u,c)) \right] \sm \{e,f\}\right|}
		\\&\quad \cdot \left(1-\frac{\p}{L_i}\right)^{\left|\left[N^{+}(v,x) \cup N^{+}(w,c)\cup N^-(u,x)\cup N^-(u,c) \right] \sm [(N^-(u,x)\cap N^-(u,c)) \cup \{e,f\}]\right|}
		\\&\quad \cdot (1-\Eq(e,x)) (1-\Eq(f,c))
		\\&\leq \left(\frac{\p}{L_i}\right)^2
		\\&\quad \cdot \left(1-\frac{\p}{L_i}\right)^{\left|\left[N^{+}(v,x) \cup N^{+}(w,c)\cup N^-(u,x)\cup N^-(u,c) \right] \sm \{e,f\}| +|[N^-(u,x)\cap N^-(u,c)] \sm \{e,f\}\right|}
		\\&\quad \cdot (1-\Eq(e,x)) (1-\Eq(f,c)).
	\end{align*}
	Note that
	\begin{align*}
		&\quad |\left[N^{+}(v,x) \cup N^{+}(w,c)\cup N^-(u,x)\cup N^-(u,c) \right] \sm \{e,f\}|  
		\\&\quad +|[N^-(u,x)\cap N^-(u,c)] \sm \{e,f\}| 
		\\&= |N^{+}(v,x)\sm \{e\}| + |N^-(u,x)\sm \{e,f\}| + |N^-(u,c)\sm \{e,f\}| +  |N^{+}(w,c)\sm \{f\}|
		\\&\geq |N^{+}(v,x)| + |N^-(u,x) | + |N^-(u,c) | +  |N^{+}(w,c) | - 6.
	\end{align*}
	Hence, following the computation in \eqref{eq:retain} we obtain
	\begin{align*}
		\Pr(Z(x,f)) \leq \left(\frac{\p}{L_i} \right)^2 \left(1-\frac{\p}{L_i}\right)^{-2}  \R_i^4 = \frac{\p^2}{(L_i -\p)^2} \R_i^4.
	\end{align*}
	Since the events $Z(x,f)$ are disjoint, we have:
	\begin{align*}
		\Pr(\overline{A} \cap \overline{B}) &= \frac{\p}{L_i} \R_i^2 + \sum_{x \in L(e)\sm\{c\},f\in N^-(u,c)\sm \{e\}}  \Pr(Z(x,f))  \\
		&\leq \frac{\p}{L_i} \R_i^2 + (L_i -1) (|N^-(u,c)| -1) \frac{\p^2}{(L_i -\p)^2} \R_i^4  \\
		&< \frac{1}{L_i} + \p^2\frac{|N^-(u,c)|}{L_i}\R_i^4.   \qedhere
	\end{align*} 
\end{proof}
\subsection{Concentration}\label{sec:MR-concentration}
We use Talagrand's inequality to show the concentration.
\begin{theorem}[Talagrand's inequality~\cite{Tal95}]\label{thm:talagrand}
	Let $X$ be a non-negative random variable determined by the independent trials $T_1,\ldots,T_n$. Suppose that for every set of possible outcomes of the trials, we have:
	\begin{enumerate}[(i)]
		\item changing the outcome of any one trial can affect $X$ by at most $k$ and
		\item for each $s > 0$, if $X \geq s$ then there is a set of at most $rs$ trials whose outcomes certify $X \geq s$.
	\end{enumerate}
	Then for any $t > 96k \sqrt{r \Exp(X)} + 128rk^2$ we have
	\begin{align*}
		\Pr(|X-\Exp(X)| > t) \leq 4 \exp\left({-\frac{t^2}{8k^2r(4\Exp(X)+t)}}\right).
	\end{align*}
\end{theorem}

The following three claims contain the desired concentration bounds.

\begin{claim}[{\cite[Claim 26]{LP23}}]\label{cla:concentration-R}
	For every vertex $v$ and color $c$, we have
	\begin{align*} 	
		\Pr(\left||R'^{+}(v,c)|-\Exp(|R'^{+}(v,c)|)\right|> \tfrac{1}{2} \sqrt{R_i} \log^2 \D) \leq \frac{\D^{-10\log \D}}{5},
		\\ \Pr(\left||R'^-(v,c)|-\Exp(|R'^-(v,c)|)\right|> \tfrac{1}{2} \sqrt{R_i} \log^2 \D) \leq \frac{\D^{-10\log \D}}{5}.
	\end{align*}
\end{claim}
\begin{proof}
	We prove only for $R'^{+}(v,c)$ here, as the proof for $R'^-(v,c)$ is similar.
	Fix a vertex $v$, a color $c$ and let $R'^{+} = |R'^{+}(v,c)|$.
	We wish to apply Talagrand's inequality to $R'^{+}$ with $k=2$, $r=1$ and $t=\tfrac{1}{2}\sqrt{R_i} \log^2\D$.
	
	It easy to see that changing any of the outcome of the arc activation in step~\ref{step:activate1}, color assignment in step~\ref{step:assign1} or coin flip in step~\ref{step:unassign-and-coinflip1} and~\ref{step:update-lists-and-coinflip1} can affect the size of $R'^{+}$ by at most $2$.
	(Here it is important to recall that at most one arc of $R^{+}(v,c)$ can retain a specific color at the same time.)
	Moreover, for every arc $e=vu \in R'^{+}(v,c)$, there is either an activation event or an assignment of a color to an arc $vx$ or $wu$ that witnesses $e$ being still uncolored after step~\ref{step:unassign-and-coinflip1}, so $r=1$ is enough.
	Recall that by Lemma~\ref{lem:sizes}, $R_i \geq \log^7 \D$.
	Moreover, $\Exp(R'^{+}) \leq R_i $.
	Since $\D$ is large, we have $$t=\tfrac{1}{2}\sqrt{R_i} \log^2\D > 96k \sqrt{r \Exp(R'^{+})} + 128rk^2.$$ Therefore we obtain from Theorem~\ref{thm:talagrand} that 
	\begin{align*}
		\Pr(|R'^{+}-\Exp(R'^{+})| > t) \leq 4 \exp\left({-\frac{t^2}{8k^2r(4\Exp(R'^{+})+t)}}\right)\leq 4\exp\left(-\frac{t^2}{2^8R_i}\right) \leq \frac{\D^{-10\log \D}}{5},
	\end{align*}
	where $t \leq 4\Exp(R'^{+}) \leq 4R_i$ is used in the second inequality.
\end{proof}
\begin{claim}[{\cite[Claim 27]{LP23}}]\label{cla:concentration-L}
	For every arc $e$, we have 
	\begin{align*} 	
		\Pr(\left||L'(e)|-\Exp(|L'(e)|)\right|> \tfrac{1}{2} \sqrt{L_i} \log^2 \D) \leq \frac{\D^{-10\log \D}}{5}.
	\end{align*}
\end{claim}
\begin{proof}
	Fix an arc $e=uv$ and let $L' = |L'(e)|$.
	Let $X$ be the number of colors $c \in L(e)$, which are retained by at least one arc in $(N^{+}(u,c) \cup N^-(v,c))\sm\{e\}$.
	For $0 \leq n \leq m \leq 2$, we define $Y_{m,n}$ to be the number of colors which are assigned to an arc in \emph{exactly} $m$ of $N^{+}(u,c)\sm \{e\}$, $N^-(v,c)\sm \{e\}$ and are removed from an arc in at least $n$ of $N^{+}(u,c)\sm \{e\}$, $N^-(v,c)\sm \{e\}$ during step~\ref{step:assign1} and~\ref{step:unassign-and-coinflip1} of the procedure.
	Similarly, we define $X_{m,n}$ to be the number of colors which are assigned to an arc in \emph{at least} $m$ of $N^{+}(u,c)\sm \{e\}$, $N^-(v,c)\sm \{e\}$ and are removed from an arc in at least $n$ of $N^{+}(u,c)\sm \{e\}$, $N^-(v,c)\sm \{e\}$ during step~\ref{step:assign1} and~\ref{step:unassign-and-coinflip1} of the procedure.
	Note that $Y_{2,n} = X_{2,n}$ and for $m<2$, $Y_{m,n} = X_{m,n} - X_{m+1,n}$.
	Making use of the very useful fact that for any $v$, if in step~\ref{step:unassign-and-coinflip1} a color $c$ is removed from at least one arc in $N^-(v,c)$ ($N^{+}(u,c)$ respectively), then it is removed from every arc in $N^-(v,c)$ ($N^{+}(u,c)$ respectively), we obtain that:
	\begin{align*}
		X=(Y_{2,0} - Y_{2,2}) + (Y_{1,0} - Y_{1,1}) = (X_{2,0} - X_{2,2}) + ((X_{1,0} - X_{2,0})-(X_{1,1} - X_{2,1})).
	\end{align*}
	
	Fix $1\leq m,n\leq 2$. We will show that $X_{m,n}$ is highly concentrated
	by applying Talagrand's inequality (Theorem~\ref{thm:talagrand}) to $X_{m,n}$ with $k=2$, $r=4$ and $t=  \tfrac{1}{14}\sqrt{L_i} \log^2\D$.
	First we check $k=2$. Note that changing the arc activation can only affect whether one color is counted by $X_{m,n}$, 
    changing the color assigned to an arc from $c_1$ to $c_2$ can only affect whether $c_1$ or $c_2$ are counted by $X_{m,n}$, and changing the decision to uncolor an arc in step~\ref{step:unassign-and-coinflip1} can only affect whether the color of the arc is counted by $X_{m,n}$.
	Secondly we check $r=4$. If $X_{m,n} \geq s$, then there is a set of at most $s(m+n)\le 4s$ outcomes which certify this fact, namely for each of the $s$ colors, $m$ arcs on which that color appears, along with $n$ (or fewer) outcomes which cause $n$ of those arcs to be uncolored.
	Recall that by Lemma~\ref{lem:sizes}, $L_i \geq \log^7 \D$.
	Moreover, $\Exp(X_{m,n}) \leq  L_i$.
	Since $$t=\tfrac{1}{14}\sqrt{L_i} \log^2\D > 96k \sqrt{r \Exp(X_{m,n})} + 128rk^2,$$ we obtain from Theorem~\ref{thm:talagrand} that 
	\begin{align*}
		\Pr(|X_{m,n}-\Exp(X_{m,n})| > t) \leq 4 \exp\left({-\frac{t^2}{8k^2r(4\Exp(X_{m,n})+t)}}\right)\leq \exp\left(-\frac{t^2}{2^{11} L_i}\right) \leq \frac{\D^{-10\log \D}}{35},
	\end{align*}
	where we used that $t, 4\Exp(X_{m,n}) \leq 4L_i$ in the second inequality.
	
	Finally, let $X'$ be the number of colors $c$ removed from $\bigcup_{f \in N^{+}(u,c)} L(f)$ or  $\bigcup_{f \in N^-(v,c)} L(f)$ in step~\ref{step:update-lists-and-coinflip1} by coin flips.
	It follows easily from the Chernoff Bound that $X'$ is  highly concentrated around its expectation.
	Since $|L'(e)| = |L(e)| - (X + X')$, this completes the proof.
\end{proof}
\begin{claim}[{\cite[Claim 28]{LP23}}]
	For every vertex $v$ and color $c$, we have
	\begin{align*} 	
		\Pr(\left||N^{*+}(v,c)|-\Exp(|N^{*+}(v,c)|)\right|> \tfrac{1}{2} \sqrt{N_i} \log^2 \D) \leq \frac{\D^{-10\log \D}}{5},
		\\ \Pr(\left||N^{*-}(v,c)|-\Exp(|N^{*-}(v,c)|)\right|> \tfrac{1}{2} \sqrt{N_i} \log^2 \D) \leq \frac{\D^{-10\log \D}}{5}.
	\end{align*}
\end{claim}
\begin{proof}
	We prove only for $N^{*+}(v,c)$. The proof for $N^{*-}(v,c)$ is strictly symmetric and requires no additional arguments.
	We again proceed indirectly.
	Let $A_{v,c}$ be the set of arcs $e \in N^{+}(v,c)$ that are still uncolored after step~\ref{step:unassign-and-coinflip1}.
	Let $B_{v,c}$ be the set of arcs $e = vu$ in $A_{v,c}$ such that $c$ is retained on some arc of $N^-(u,c)$.
	Let $C_{v,c}$ be the set of arcs $e =vu$ in $A_{v,c} \sm B_{v,c}$ such that $c$ is removed from  the lists $L(f)$ of all arcs $f \in N^{-}(u,c)$ because of the coin flip in step~\ref{step:update-lists-and-coinflip1}.
	
	The proof that $|A_{v,c}|$ is highly concentrated is virtually identical to the proof Claim~\ref{cla:concentration-R}.
	The proof that $|B_{v,c}|$ and $|C_{v,c}|$ are highly concentrated follows along the lines of the proof of Claim~\ref{cla:concentration-L}.
	Since $N^{*+}(v,c) = A_{v,c} \sm (B_{v,c} \cup C_{v,c})$, the desired result follows.
\end{proof}

\section{Proofs of Claims~\ref{cla:cycle-tweak} and \ref{cla:simultaneously}}\label{sec:cycle}
We give here the proofs of Claims~\ref{cla:cycle-tweak} and \ref{cla:simultaneously}, which can also be found in Sections~$4$ and $5$ of Lang and Postle~\cite{LP23}. The main difference is that the suspicious paths we considered are directed. However, this lead to only a few changes in the proof.
The following arguments are therefore almost identical to the ones of Lang and Postle.

\subsection{Proof of Claim~\ref{cla:cycle-tweak}}\label{sec:cycle-tweak}
The proof of Claim~\ref{cla:cycle-tweak}, similar to that of Claim~\ref{cla:MR-edge-bounds}, is still a concentration analysis.
First we show that the expectation of $|X(e)|$, $|Y^{+}(v,c)|$ and $|Y^-(v,c)|$ are bounded by $4\p$ and the expectation of $|Z^{+}(v,c)|$ and $|Z^-(v,c)|$ are bounded by $4\p \log \D$.
Then we show that with high probability, neither of those random variables deviates by more than $\tfrac{1}{4}\sqrt{L_i} \log^2 \D$ ($\tfrac{1}{4}\sqrt{N_i}\log^2 \D$, $\tfrac{1}{4}\sqrt{R_i}\log^2 \D$ respectively) from their expectation.

The next observation bounds the probability of a suspicious directed path becoming monochromatic after step~\ref{step:assign1} of our random coloring procedure.
\begin{observation}\label{obs:probability-monochromatic}
	Let $\gamma$ be a partial $(1,1)$-edge-coloring of $D$ from some lists $L(e)\subset \cL(e)$.
    Let $L$ be an integer such that $|L(e) |= L$ for every arc $e$.
	Let  $\gamma'$ be an extension of $\gamma$ obtained as follows:
	\begin{enumerate}[\upshape(I)]
		\item Activate each uncolored arc with probability $\p$.
		\item Assign to each activated arc $e$ a color chosen uniformly at random from $L(e)$.
	\end{enumerate}
	Then the probability that a suspicious directed path $P \in P_{L,\gamma}(u,v,c;k)$ becomes monochromatic with $c$ under $\gamma'$ is $(\p/L)^{k}$.
\end{observation}
\begin{proof} 
	The probability that an arc $e$ is activated and assigned some color $\alpha \in L(e)$ is $\p/L$.
	Since activations and color assignments are independent and $P$ has uncolored length $k$, the observation follows.
\end{proof}

We now prove Claim~\ref{cla:cycle-tweak}.
\subsubsection{Expectation}\label{sec:XYZExpectation}
Let us start with the expectation of $X(e)$ for an arc $e=uv$.
Recall that this is the set of colors $c\in L(uv)$ for which there is a directed path in $P(v,u,c)$ that becomes monochromatic with $c$ after step~\ref{step:assign1}.
By~\eqref{equ:monochromatic-paths}, such a directed path $P$ is in $P(u,c;\ell)$ or $P(v,u,c;k)$ for some $1 \leq k \leq \ell-1$.
Since $\gamma$ is a $(1,1)$-edge-coloring and by~\ref{itm:color-neighbors-size-before} of Claim~\ref{cla:nibbler}, we can use Observation~\ref{obs:number-suspicious} to  bound $|P(u,c;\ell)| \leq N_i^{\ell}$ and $|P(v,u,c;k)| \leq N_i^{k-1}$.
By Observation~\ref{obs:probability-monochromatic}, the probability that $P \in P(v,u,c;k)$ becomes monochromatic with $c$ after step~\ref{step:assign1} is $(\p/ L_i)^{k}$.
Similarly, the probability that $P \in P(v,c;\ell)$ becomes monochromatic with $c$ after step~\ref{step:assign1} is $(\p/ L_i)^{\ell}$ .
Also note that by the choice of $p=1/4$ and $\ell=2\log \D$ in Setup~\ref{Setup}, we have
\begin{align*}
	\p^{\l} = \p\cdot \frac{1}{4^{2 \log \D -1}}  \leq \frac{\p}{2\D} \leq  \frac{\p}{L_i}.
\end{align*}
Therefore, we can bound the probability that a directed path in $P(v,u,c)$ becomes monochromatic with $c$ after step~\ref{step:assign} by
\begin{align*}
	\left(\frac{\p N_i}{L_i}  \right)^{\l} +    \frac{\p}{L_i} \sum_{k=0}^{{\l}-2} \left(\frac{\p N_i}{L_i}\right)^k 
	&\leq \p^{{\l}} +  \frac{\p}{L_i} \cdot \frac{1}{1-p}  
	\leq \frac{4\p}{L_i},
\end{align*}
where we use $N_i\le L_i$, which holds by Lemma~\ref{lem:sizes}.
Linearity of expectation then gives  $\Exp(|X(e)|) \leq 4\p$.
The expectations of $|Y^{+}(v,c)|$, $|Y^-(v,c)|$, $|Z^{+}(v,c)|$ and $|Z^-(v,c)|$ follow analogously, since $|N^{+}(v,c)|$, $|N^-(v,c)| \leq N_i$, $|R^{+}(v,c)|$, $|R^-(v,c)| \leq R_i$ by assumption and $N_i \leq L_i$ and $R_i / L_i \leq \log \D$ by Lemma~\ref{lem:sizes}.

\subsubsection{Concentration}
We use Talagrand's inequality to show that the random variables $|X(e)|$, $|Y^{+}(v,c)|$, $|Y^-(v,c)|$, $|Z^{+}(v,c)|$, $|Z^-(v,c)|$ are highly concentrated around their expectation. 

Fix an arc $e=uv$ and let us abbreviate $X = |X(e)|$.
Our intention is to apply Talagrand's inequality to $X$ with $k=2$, $r=\l$ and $t=\tfrac{1}{4}\sqrt{L_i} \log^2\D$.
Note that changing any of the outcome of the arc activation in step~\ref{step:activate1} or color assignment in step~\ref{step:assign1} can affect the size of $X$ by at most $2$.
This is because at most one color might be added to or removed from $X(e)$ this way.
Moreover, for every color $c \in X(e)$, there must be a directed path in $P \in   P(v,u,c)$ together with at most $\l$ arcs of $P$ that have been activated and assigned $c$. So $r=\ell$ is enough.
By the assumption $i < i_0$, we have that $L_i \geq \log^7 \D$.
Moreover, we have shown $\Exp(X) \leq 4\p$ in Appendix~\ref{sec:XYZExpectation}.
Now since $\D$ is large enough, we have
$$t=\tfrac{1}{4}\sqrt{L_i} \log^2\D \geq \log^5 \D>192 \sqrt{4\p\l} + 512\l \geq 96k \sqrt{r \Exp(X)} + 128rk^2.$$
Hence we obtain from Talagrand's inequality (Theorem~\ref{thm:talagrand}) that 
\begin{align*}
	\Pr(|X-\Exp(X)| > t) \leq 4 \exp\left({-\frac{t^2}{8k^2r(4\Exp(X)+t)}}\right) \leq   \exp\left(-\frac{\sqrt{L_i}\log \D}{2^ {10}}\right) \leq \frac{\D^{-10\log \D}}{5},
\end{align*}
where we used $4\Exp(X)\leq t$ and $r=\ell = 2 \log \D$ in the second inequality.
The concentration of $Y^{+}(v,c)$, $Y^-(v,c)$, $Z^{+}(v,c)$ and $Z^-(v,c)$ follows along the same lines.

\subsection{Proof of Claim~\ref{cla:simultaneously}}\label{sec:uniform-bounds}
We use the Lov\'asz Local Lemma (Lemma~\ref{lem:LLL}) to prove Claim~\ref{cla:simultaneously}.
The following observation follows in the same fashion as Observation~\ref{obs:number-suspicious} (using that $\gamma$ is a $(1,1)$-edge-coloring and conditions~\ref{itm:list-size-before},~\ref{itm:color-neighbors-size-before} of Claim~\ref{cla:nibbler}).
We omit the proof.
\begin{observation}\label{obs:suspicous-path-for-c}
	For every uncolored arc $e$, there are at most $L_iN_i^k$ directed paths of uncolored length at most $k$ which start with an uncolored arc, end with $e$ (start with $e$, end with an uncolored arc) and are suspicious for some color $c \in L(e)$.
\end{observation}

We now prove Claim~\ref{cla:simultaneously}.
\begin{proof}[Proof of Claim~\ref{cla:simultaneously}]
	We wish to apply the  Lov\'asz Local Lemma,
	so we need to analyze the dependencies between the random variables  $L'(e)$ , $N'^{+}(v,c)$, $N'^-(v,c)$, $R'^{+}(v,c)$,  $R'^-(v,c)$,  $X(e)$,  $Y^{+}(v,c)$,  $Y^-(v,c)$, $Z^{+}(v,c)$ and $Z^-(v,c)$.
	
	The variables $L'(e)$, $N'^{+}(v,c)$, $N'^-(v,c)$, $R'^{+}(v,c)$ and $R'^-(v,c)$ are determined by the activations and color assignments to arcs of distance at most $2$ to $v$ or $e$ in the digraph $D$.
	Hence random variables of type $L'(e)$, $N'^{+}(v,c)$, $N'^-(v,c)$, $R'^{+}(v,c)$ and  $R'^-(v,c)$ are each independent of all but at most $(8 N_iL_iR_i)^2$ random variables of type $L'(\cdot)$ , $N'^{+}(\cdot,\cdot)$, $N'^-(\cdot,\cdot)$, $R'^{+}(\cdot,\cdot)$ and $R'^-(\cdot,\cdot)$.

	The dependency relations associated with random variables of type $X(e)$, $Y^{+}(v,c)$, $Y^-(v,c)$, $Z^{+}(v,c)$ and $Z^-(v,c)$ are a bit more delicate.
	This is because the events determining these variables are not constrained by graph distance.
	
	Consider an uncolored arc $e=uv$.
	In the following, we bound the number of random variables of type $X(\cdot)$ that $X(e)$ is not independent of.
    If there is a directed path in $P(v,u,c)$ that becomes monochromatic with $c$, then this path together with $uv$ can be seen as a suspicious path ending with $uv$, and of uncolored length at most $\ell + 1$.
	Let $f$ be another uncolored arc.
	Then the random variables $X(e)$ and $X(f)$ are independent unless the following holds.
	There are suspicious paths $P,P'$ for colors $c \in L(e)$, $c'\in L(f)$ that end with $e$, $f$, respectively.
	Moreover, $P$ and $P'$ have each at most $\ell+1$ uncolored arcs and cross in an uncolored arc $g$. By considering subpaths of $P$ and $P'$ starting with $g$, we know there are suspicious paths from $g$ to $e$ and from $g$ to $f$.
	In light of Observation~\ref{obs:suspicous-path-for-c}, it follows that $X(e)$ is independent of all but at most $(L_iN_i^{\ell+1})^2$ random variables of type  $X(\cdot)$.
	Note that the first factor $L_iN_i^{\ell+1}$ counts the suspicious directed paths ending with $e$ and starting with an uncolored arc. Let $g$ denote this uncolored arc. Then the second factor counts the suspicious directed paths starting with $g$ and ending with an uncolored arc.
	
	In the same way, we can argue that $X(e)$ is independent of all but at most 
    $(L_iN_i^{\ell+1})^2$
    random variables of type $Y^-(v,c)$ ($Y^{+}(v,c)$ respectively).
	We also obtain that $X(e)$ is independent of all but at most $(L_iN_i^{\ell+1})(L_iN_i^{\ell}R_i)$ random variables of type  $Z^{+}(v,c)$ ($Z^-(v,c)$ respectively).
    Here, the second factor counts the at most $L_iN^{\ell}$ suspicious directed paths ending in one of the at most $R_i$ arcs of $R^-(v,c)$ ($R^{+}(v,c)$ respectively).
	Finally, we see that $X(e)$ is independent of all but at most $(L_iN_i^{\ell+1})(8N_iL_iR_i)^2$ 
    random variables of type $L'(f)$, $N'^{+}(v,c)$, $N'^-(v,c)$, $R'^{+}(v,c)$ and $R'^-(v,c)$.
	This concludes the analysis of the dependency relations for the random variable $X(e)$.
	
	Analogous statements can be made from the perspective of random variables $L'(e)$ , $N'^{+}(v,c)$, $N'^-(v,c)$, $R'^{+}(v,c)$,  $R'^-(v,c)$,  $Y^{+}(v,c)$, $Y^-(v,c)$, $Z^{+}(v,c)$ and $Z^-(v,c)$.
	Since the arguments are identical, we omit the details.
	In summary, each of the discussed random variables is independent of all but at most 
	$$d=10(8 L_i^2 N_i^{\ell+1} R_i^2)^2 \leq \D^{4\ell} = \D^{8\log \D} \ll \D^{10\log \D}$$ 
	of the other random variables.
	Hence we can apply the Lov\'asz Local Lemma (Lemma~\ref{lem:LLL}) with $d$ and $p=\D^{-10\log \D}$. 
	It follows that with positive probability the conclusions of Claims~\ref{cla:MR-edge-bounds}  and~\ref{cla:cycle-tweak} hold simultaneously for all arcs $e$, vertices $v$ and colors $c$.
\end{proof}

\end{document}